\documentclass[]{amsart}
\usepackage{amssymb,amsmath} 
\usepackage[mathscr]{euscript}

\newcounter{sec}

\newcounter{punct}[sec]

\def\punct{\refstepcounter{punct}{\arabic{sec}.\arabic{punct}.  }}

\newtheorem{theorem}{Theorem}[sec]
\newtheorem{proposition}[theorem]{Proposition}

\newtheorem{lemma}[theorem]{Lemma}

\newtheorem{corollary}[theorem]{Corollary}
\newtheorem{observation}[theorem]{Observation}

\def\COUNTERS{\addtocounter{sec}{1}
              \setcounter{punct}{0}
          \setcounter{equation}{0}
          \setcounter{theorem}{0}
          }

\def\SL{\mathrm {SL}}
\def\SU{\mathrm {SU}}
\def\GL{\mathrm  {GL}}

\def\SO{\mathrm  {SO}}

\def\phi{\varphi}
\def\epsilon{\varepsilon}
\def\kappa{\varkappa}

\renewcommand\O{\mathrm{O}}

\def\le{\leqslant}
\def\ge{\geqslant}

\renewcommand{\Re}{\mathop{\rm Re}\nolimits}

\renewcommand{\Im}{\mathop{\rm Im}\nolimits}

\def\la{\langle}
\def\ra{\rangle}

\def\cB{\EuScript B}
\def\cC{\EuScript C}

\def\cH{\EuScript H}

\def\cL{\EuScript L}

\def\cR{\EuScript R}
\def\cS{\boldsymbol{\EuScript S}}

\def\cU{\EuScript U}

\def\cW{\EuScript W}

\def\frF{\mathfrak F}

\def\frf{\mathfrak f}

\def\frl{\mathfrak l}

\def\fro{\mathfrak o}

\def\frs{\mathfrak s}

\def\R {{\mathbb R }}
 \def\C {{\mathbb C }}
  \def\Z{{\mathbb Z}}

 \def\ov{\overline}

\def\wh{\widehat}

\def\sm{\smallskip}

\def\0{{\ov 0}}
\def\1{{\ov 1}}

\def\DeltA{\boldsymbol{\Delta}}

\begin{document}

\begin{center}
\bf \Large
Spherical harmonics,
\\
 operators of multiplication by coordinates,
 \\
  and infinitesimal conformal transformations

\bigskip

\sc \large

Yury A. Neretin%
\footnote{Supported by the grant FWF PAT5335224.}
\end{center}

\bigskip

{\small Consider the space of $C^\infty$-functions on the two-dimensional sphere $S^2$ and its decomposition $\oplus\mathcal H_n$ into a direct sum of minimal rotation-invariant spaces. 
 We consider elements of $\oplus\mathcal H_n$
 as functions of two variables, a nonnegative integer variable $n$
and a complex variable $u$ (a restriction of such function to the set $n=k$ is a polynomial in  $u$ of degree
$\le 2k$).  For operators of multiplication by  $x_1$, $x_2$, $x_3$ in $C^\infty(S^2)$ we obtain the corresponding operators in 
$\oplus\mathcal H_n$, they are differential-difference operators
in the variables $u$, $n$ (including second derivatives in $u$ and shifts $n\mapsto n\pm1$). We obtain the similar correspondence for operators of differentiation
along conformal vector fields on $S^2$. 
}

\section{Introduction}

\COUNTERS

{\bf \punct Spherical harmonics.}
We consider the sphere 
$$
x_1^2+x_2^2+x_3^2=1
$$
in $\R^3$ and the action of the  orthogonal
group $\O(3)$ in the space $C^\infty(S^2)$ of
smooth functions on the sphere.  
We get a direct sum of  
 minimal $\O(3)$-invariant subspaces
 (spherical harmonics) 
$$
\bigoplus_{n=0}^\infty \cH_n,\quad  
\text{where $\dim \cH_n={2n+1}$.}
$$

 It is a well-known (and nontrivial) topic of   theory of special functions and classical analysis in $\R^3$. It is also a topic
of representation theory, which at the first glance looks like a toy.

Theory of spherical harmonics was exposed in books by Hobson \cite{Hob};
MacRobert \cite{Mac};
Landau, Lifshitz \cite{LL};
Erd\'elyi, Magnus, Oberhettinger, Tricomi \cite{HTF2};
Gelfand, Minlos, Shapiro \cite{GMS}; 
Tychonov,  Samarski \cite{TS}; 
Vilenkin \cite{Vil};
Stein, Weiss \cite{SW}; 
Perelomov \cite{Per}; 
Helgason \cite{Hel}; 
Groemer \cite{Gro};
Terras \cite{Ter}; 
Andrews, Askey, Roy \cite{AAR};
Atkinson, Han \cite{AH}. 
These expositions overlap but do not cover one another.
We discuss a topic outside these works.

\sm

{\bf\punct The kernel.} 
We define the following function 
\begin{equation}
K(x_1,x_2,x_3; u, n):=
\biggl(\frac{\bigl(1-x_3+u(x_1-ix_2)\bigr)
\bigl(x_1+ix_2+u(x_3-1)\bigr)}{2(1-x_3)}\biggr)^n,
\label{eq:K-x}
\end{equation}
where $(x_1,x_2,x_3)$ ranges in the sphere $S^2$,
$n$ ranges in the space of nonnegative integers, and $u$ ranges in $\C$.

Passing to the spherical coordinates 
$$
x_1=\cos\theta\sin\psi,\quad x_2=\cos\theta \cos\psi,
\quad
x_3=\sin\theta
$$
where $\psi\in [0,2\pi]$, $\theta\in [-\pi/2,\pi/2]$,
we come to the expression
\begin{equation}
K(\psi,\theta;u,n)=
\Bigl(\frac 12\cdot \bigl(i e^{-i\psi}\cos\theta
+2u \sin \theta+ i u^2 e^{i\psi}\cos\theta
\bigr)
\Bigr)^n
\label{eq:K-phi}
\end{equation}

Below we prefer to regard $S^2$ as the Riemann sphere
and pass to the stereographic coordinates
$$
z=\frac{x_1+ix_2}{1-x_3}\,\in\C,
$$ 
the kernel in these coordinates has the form
$$
K(z, \ov z; u,n):=\Bigl(\frac{(1+\ov z u)(z-u}{1+z\ov z}\Bigr)^n
.$$

For each $n$ we consider the linear operator
$$
A_n f(u)=\int_{S^2} K(x; u,n) \,f(x)\,d\sigma(x)
$$
from the space of $C^\infty(S^2)$
to the space of polynomials in $u$ of degree
$\le 2n$. By $d\sigma(x)$ we denote the probabilistic
Lebesgue measure on the sphere. 

In the stereographic coordinates these operators have the form
$$
A_n f(u)=\frac 1\pi \int_\C K(z,\ov z;u,n)\, f(z,\ov z)\,\frac{d\Re z\,d\Im z}{(1+z\ov z)^2}.
$$
These operators are intertwining operators from $C^\infty(S^2)$
to $(2n+1)$-dimensional irreducible representations of $\SO(3)$
(see below Proposition \ref{pr:intertwiner}).
We consider all these functions together
regarding the collection $\{\frf_n\}:=\{A_n f\}$, where $n\ge 0$, as a function 
$\frF(u,n)$ depending
on a nonnegative integer $n$ and a complex variable $u$, 
$$
\frF(u,n):=\frf_n(u).
$$
Denote by $\cW^\infty$ the space of functions in $(u,n)$
obtained in this way. Its explicit description is given below
in Subsection \ref{ss:image-C-infty}.

\sm

{\sc Remark.} The intertwining operator $C^\infty(S^2)\to \cH_n$ is defined canonically up to a constant factor depending on $n$, and all ways to write 
this operator leads to the same
result, see for example, \cite{HTF2}, Sect. II.5.1.
\hfill $\boxtimes$

\sm

{\bf \punct The statement of the paper.} We work in the stereographic coordinates $z$, $\ov z$ on the sphere. 
It is easy to see that 
the transformation $f\mapsto \frF$ establishes the following
correspondence of  operators
in $C^\infty(S^2)$ and $\cW^\infty$:
\begin{align}
\frac{\partial}{\partial z}+\ov z^2 \frac{\partial}{\partial \ov z}
\quad&\longleftrightarrow\quad \frac{\partial}{\partial u}
;
\label{eq:triv1}
\\
z  \frac{\partial}{\partial z}- \ov z \frac{\partial}{\partial \ov z}
\quad&\longleftrightarrow\quad
u\frac{\partial}{\partial u}+n
;
\label{eq:triv2}
\\
z^2\frac{\partial}{\partial z}+ \frac{\partial}{\partial \ov z}
\quad&\longleftrightarrow\quad
u^2 \frac{\partial}{\partial u}+2n u
.
\label{eq:triv3}
\end{align} 
This is simply the intertwining property (the operator $A_n$ commutes with the Lie
algebra of the orthogonal group).

Denote by $T_+$ and $T_-$ the shift operators
in $\cW^\infty$:
$$
T_\pm\, \frF(u,n)=\frF(u, n\pm 1),\qquad T_-=T_+^{-1}.
$$ 

\medskip

First, consider the operators of multiplication by the following
functions:
$$
\frac{\ov z}{1+z \ov z},\qquad
\frac{1- z\ov z}{1+z \ov z},\qquad
\frac{z}{1+z \ov z}.
$$
In coordinates $(x_1,x_2,x_3)$ these operators are multiplications
by the functions 
$$
x_1-i x_2,\qquad x_3,\qquad x_1+i x_2.
$$

\begin{theorem}
\label{th:1}
We have the following correspondence of operators
in $C^\infty(S^2)$ and $\cW^\infty$:
\begin{align}
\frac{\ov z}{1+z \ov z}\quad&\longleftrightarrow\quad
-\frac1{2(n+1)(2n+1)}\, T_+\, \frac{\partial^2}{\partial u^2}
+\frac{n}{2(2n+1)}\,T_-
;
\label{eq:drob1}
\\
\frac{1- z\ov z}{1+z \ov z}\quad&\longleftrightarrow\quad
\frac1{(n+1)(2n+1)} \, T_+\, u \frac{\partial^2}{\partial u^2}
+\frac 1{n+1}\, T_+ \frac{\partial}{\partial u}
-\frac{n}{2n+1}\, T_- \,u
;
\label{eq:drob2}
\\
\frac{z}{1+z \ov z} \quad&\longleftrightarrow\quad
-\frac{1}{2(n+1)(2n+1)}\, T_+\, u^2 \frac{\partial^2}{\partial u^2}
+\frac{2}{(n+1)}\, T_+\, u \frac{\partial}{\partial u}-
\label{eq:drob3}
\\
\quad&\phantom{\longleftrightarrow}\quad
\qquad\qquad\qquad\qquad\qquad\qquad\qquad
-2T_+ 
- \frac n{2(2n+1)} \,T_- \, u^2
\notag
.
\end{align}
\end{theorem}

  
Next, we present operators in $\cW^\infty$ corresponding
to conformal vector fields on the sphere.    

\begin{theorem}
\label{th:2}
We have the following correspondences of operators
in $C^\infty(S^2)$ and $\cW^\infty$:
\begin{equation}
\frac{\partial}{\partial z}- \ov z^2\frac{\partial}{\partial \ov z}
\quad\longleftrightarrow\quad 
-\frac{n + 2}{(n+1)(2 n+1)}\,
T_+ \,\frac{\partial^2}{\partial u^2}  - 
 \frac{n (n -  1)}{(2 n+1)}\, T_-  ;
\label{eq:conf1}
\end{equation}
\begin{multline}
z\frac{\partial}{\partial z}+ \ov z\frac{\partial}{\partial \ov z}
\quad\longleftrightarrow\quad 
\\ - \frac{n+2}{(n+1) (2 n+1)}\,
 T_+\, u \frac{\partial^2}{\partial u^2}+ 
 \frac{n+2}{n+1}\, T_+ \,  \frac{\partial}{\partial u}
 -\frac{(n-1) n}{2 n+1}\, T_- \,u ;
\label{eq:conf2}
\end{multline}
\begin{multline}
z^2\frac{\partial}{\partial z}- \frac{\partial}{\partial \ov z}
\quad\longleftrightarrow\quad 
-\frac{(n+2) }{(n+1) (2 n+1)}\,\, T_+ \, u^2\frac{\partial^2}{\partial u^2}
+\\+
\frac{2 (n+2)}{n+1}\, T_+ \,u  \frac{\partial}{\partial u}
-
2 (n+2)\, T_+
-\frac{n-1}{2 n+1} \,T_-\, u^2.
\label{eq:conf3}
\end{multline}
\end{theorem}

\medskip

Further, denote by $\cS$
the following set of rational functions in $z$, $\ov z$:
\begin{equation}
\Theta^+_{n,k}=\frac{z^k}{(1+z\ov z)^n},\quad \Theta^-_{n,k}:=\frac{\ov z^k}{(1+z\ov z)^n},
\quad\text{where $0\le k\le n$.}
\end{equation}
Denote by $\C[\cS]$ their linear span.

\sm

{\sc Remark.} In coordinates $x_1$, $x_2$, $x_3$, the space $\C[\cS]$
is the space of all polynomials in $x_1$, $x_2$, $x_3$,
see below Observation \ref{obs:3}.
\hfill $\boxtimes$

\begin{corollary}
\label{cor:}
{\rm a)} For any  $0\le k\le n$ and nonnegative integers
$p$, $q$, $r$, $p'$, $q'$, $r'$ 
consider the operator
\begin{equation}
\Theta_{n,k}^\pm \, \Bigl(\frac{\partial}{\partial z}\Bigr)^p\,
\Bigl(z\frac{\partial}{\partial z}\Bigr)^q \,
\Bigl(z^2\frac{\partial}{\partial z}\Bigr)^r\,
\Bigl(\frac{\partial}{\partial \ov z}\Bigr)^{p'}\,
\Bigl(\ov z\frac{\partial}{\partial \ov z}\Bigr)^{q'} \,
\Bigl(\ov z^2\frac{\partial}{\partial \ov z}\Bigr)^{r'}.
\label{eq:NEBASIS}
\end{equation}
Then the corresponding operator in $\cW^\infty$
is a finite sum
of the form
\begin{equation}
\sum_{M,L\ge 0, K\in \Z} \Xi_{M,L,K}(n) \, T_+^K\, u^M\, \frac{\partial^L}{\partial u^L},
\label{eq:diff-diff}
\end{equation}
where $\Xi_{M,L,K}(n)$ are rational functions in $n$ with poles at
integer and half-integer points.

\sm

{\rm b)} The linear span $\cU$ of operators \eqref{eq:NEBASIS}
is an algebra, i.e., it
is closed with respect to the product.
\end{corollary}

{\sc Remark.}
We have the relation
$$ \Bigl(z\frac{\partial}{\partial z}\Bigr)^2
=
 \Bigl(\frac{\partial}{\partial z}\Bigr) \,
 \Bigl(z^2\frac{\partial}{\partial z}\Bigr)
 -
 z\frac{\partial}{\partial z}
 $$
 and the similar relation for $\frac{\partial}{\partial \ov z}$.
 Therefore, the operators \eqref{eq:NEBASIS} are linearly dependent.
 Moreover, the space $\cU$ is spanned by expressions 
 \eqref{eq:NEBASIS} with $q$, $q'$ being 0 or 1.
 \hfill $\boxtimes$.

\sm

{\bf \punct The general framework.} The statements formulated above are a special case of `operational calculus' for Plancherel decompositions extending operational calculus for the usual Fourier transform on $\R^n$. The phenomenon was observed
in \cite{Ner1} by analysing nonstandard Plancherel formulas for index hypergeometric integral transform \cite{Ner0} and difference equations for
$_3F_2[1]$. A general conjecture was formulated in \cite{Ner1}, \cite{Ner-GL2}.
Several special cases were examined by Molchanov \cite{Mol1}--\cite{Mol3} and the author \cite{Ner1}--\cite{Ner-Lob}.

Consider the spherical principal series $T_\lambda$ of representations
of the
conformal group $\O(1,3)$ of the sphere (these representations
are realized in a space
of functions on $S^2$, see below Subsection \ref{ss:conformal-sphere}).
We examine restrictions of $T_\lambda$
 to the subgroup $\O(3)$ (this is simply the decomposition
 in spherical harmonics) and write generators
of the Lie algebra
$$
\frs\fro(1,3)_\C=\frs\frl(2,\C)\oplus \frs\frl(2,\C)
$$
 in  decomposition of the restriction. We get the trivial
correspondences \eqref{eq:triv1}--\eqref{eq:triv3}
and also the following correspondences  
 \begin{multline}
\frac{\partial}{\partial z} - \ov z^2 \frac{\partial}{\partial \ov z}
+\frac{2\lambda \ov z}{1+z \ov z} \quad \longleftrightarrow
\\
 -\frac{n +\lambda+ 2}{(n+1)(2 n+1)}\,
T_+ \,\frac{\partial^2}{\partial u^2}  - 
 \frac{n (n -  \lambda - 1)}{2 n+1}\, T_-;
 \label{eq:add1} 
 \end{multline}
 \begin{multline}
z\frac{\partial}{\partial z} + \ov z \frac{\partial}{\partial \ov z}
-\frac{\lambda (1-z \ov z)}{1+z \ov z} \quad \longleftrightarrow
\\ 
- \frac{n+\lambda+2}{(n+1) (2 n+1)}\,
 T_+\, u \frac{\partial^2}{\partial u^2}+ 
 \frac{n+\lambda+2}{n+1}\, T_+ \,  \frac{\partial}{\partial u}
 -\frac{n(n-\lambda-1) }{2 n+1}\, T_- \,u;
\label{eq:add2} 
 \end{multline}
\begin{multline}
z^2\frac{\partial}{\partial z} -  \frac{\partial}{\partial \ov z}
-\frac{2\lambda z}{1+z \ov z} \quad \longleftrightarrow
\\
-\frac{(n+\lambda+2) }{(n+1) (2 n+1)}\, T_+ \, u^2\frac{\partial^2}{\partial u^2}
+\frac{2 (n+\lambda+2)}{n+1}\, T_+ \,u  \frac{\partial}{\partial u}
-
\\-2 (n+\lambda+2)\, T_+
-\frac{n-\lambda-1}{2 n+1} \,T_-\, u^2.
\label{eq:add3}
\end{multline} 

Clearly, these formulas follow from \eqref{eq:drob1}-\eqref{eq:drob3},
\eqref{eq:conf1}-\eqref{eq:conf3} and vise versa. However, it is interesting
that factors depending on $n$ in formulas \eqref{eq:add1}-\eqref{eq:add3}
have a multiplicative structure. A similar picture take place in all  situations
\cite{Mol1}--\cite{Mol3}, \cite{Ner1}--\cite{Ner-Lob} examined earlier.

Our present case is the most simple in the known zoo, but now
 the topic is the most classical. 

\def\skryt 
{\begin{align*}
-\frac{n + 2 \mu}{(1 + n)(1 + 2 n)}\,
T_+ \,\frac{\partial^2}{\partial u^2}  - 
 \frac{n (n - 2 \mu + 1)}{(1 + 2 n)}\, T_-  
 \\
- \frac{2 \mu + n}{(1 + n) (1 + 2 n)}\,
 T_+\, u \frac{\partial^2}{\partial u^2}+ 
 \frac{2 \mu + n}{1 + n}\, T_+ \,  \frac{\partial}{\partial u}
 +\frac{(-1 + 2 \mu - n) n}{1 + 2 n}\, T_- \,u
\end{align*}
\begin{multline*}
-\frac{(2 \mu + n) }{(1 + n) (1 + 2 n)}\, T_+ \, u^2\frac{\partial^2}{\partial u^2}
+\frac{2 (2 \mu + n)}{1 + n}\, T_+ \,u  \frac{\partial}{\partial u}
-
\\-2 (2 \mu + n)\, T_+
+\frac{-1 + 2 \mu - n}{1 + 2 n} \,T_-\, u^2
\end{multline*}}
 
 \sm
 
 {\bf\punct The further structure of the paper.}
 Section 2 contains preliminaries on spherical harmonics.
 Proofs of Theorems \ref{th:1}--\ref{th:2} and Corollary \ref{cor:}
 are contained in Section 3.

\section{Spherical harmonics}

\COUNTERS

{\bf \punct The stereographic projection.} 
We consider the unit sphere $S^2\subset \R^3$
$$
x_1^2+x_2^2+x_3^2=1
$$
equipped with the probabilistic measure $d\sigma$ invariant
with respect to rotations. 
Consider the {\it stereographic projection} of $S^2$ from the north pole
 $(0,0,1)$
to the coordinate plane $x_1 O x_2$. We identify this plane and
$\C$ with the coordinate $z$.
We have
\begin{equation}
z(x)=\frac{x_1+i x_2}{1-x_3}.
\label{eq:z(x)}
\end{equation}
Conversely, 
\begin{equation}
(x_1,x_2,x_3)=\Bigl(\frac{z+\ov z}{1+z\ov z},\,
 \frac{(z-\ov z)}{i(1+z\ov z)},\,
-\frac{1-z\ov z}{1+z\ov z}\Bigr).
\label{eq:x(z)}
\end{equation}

Denote by 
$$
\wh d z:= d\Re z \, d\Im z
$$
the standard Lebesgue measure on $\C$. 
The measure $d\sigma$ on $S^2$ corresponds to the measure
\begin{equation}
\frac1\pi \, \frac{\wh d z}{(1+z \ov z)^{2}}
\label{eq:measure}
\end{equation}
on $\C$.

\sm

{\bf\punct Conformal group of sphere%
\footnote{For a simple self-closed introduction
 to the groups $\SO(3)$, $\SU(2)$,
 $\SL(2,\C)$,
see for example, \cite{KM}, Sections II.11-12, for an introduction to representations of
$\SU(2)$, see for example \cite{FH}, \S\S.10.4, 11.1, see also \cite{GMS}.}.%
\label{ss:conformal-sphere}}
 Consider the Minkowski space $\R^{1,3}$
 with coordinates $(x_0,x_1,x_2,x_3)$ and the indefinite scalar product
$$
\bigl\{ (x_0,x_1,x_2,x_3),(y_0,y_1,y_2,y_3)\bigr\}
=-x_0y_0+x_1y_1+x_2y_2+x_3y_3.
$$
The group of matrices preserving this form is the pseudoorthogonal 
group $\O(1,3)$. It consists of real $(1+3)$-block matrices 
$g=\begin{pmatrix}\alpha&\beta\\\gamma&\delta\end{pmatrix}$ satisfying the condition
$$
g\begin{pmatrix}-1&0\\0&1\end{pmatrix} g^t=\begin{pmatrix}-1&0\\0&1\end{pmatrix}
$$
(here $g^t$ denotes the transposed matrix). The orthogonal group
$\O(3)$ of $\R^3$ is the subgroup in $\O(1,3)$ consisting of matrices
$\begin{pmatrix}1&0\\0&\delta\end{pmatrix}$.
The group $\O(1,3)$, consists of 4  connected components,
 the component $\SO_0(1,3)$ containing the unit is
  determined by two additional conditions:
$$
\alpha>0, \qquad \det \delta>0.
$$

Consider the cone $\cC\subset \R^{1,3}$  
$$
-x_0^2+x_1^2+x_2^2+x_3^2=0
$$
 of isotropic vectors. We identify the section $x_0=1$ of $\cC$
with the sphere $S^2$. Points of this sphere are in one-to-one correspondence
with generatrices of the cone $\cC$.
The group $\O(1,3)$ acts on the set of generatrices and therefore
it acts on $S^2$. Namely, for $x=(x_1,x_2,x_3)\in S^2$
we have
\begin{multline*}
x\mapsto \begin{pmatrix}1&x\end{pmatrix}
\mapsto
 \begin{pmatrix}1&x\end{pmatrix} \begin{pmatrix}\alpha&\beta\\\gamma&\delta\end{pmatrix}
=\begin{pmatrix}\alpha+x\gamma,\beta+x\delta\end{pmatrix}\mapsto
\\
\mapsto 
\begin{pmatrix}1&(\alpha+x\gamma)^{-1}(\beta+x\delta)\end{pmatrix}\mapsto (\alpha+x\gamma)^{-1}(\beta+x\delta)\in S^2.
\end{multline*}
Here we take a point $x$ of $S^2$ and send it to the cone $\cC$;
applying $g$ we get another point   of $\cC$;
intersecting the corresponding generatrix with the hyperplane
$x_0=1$, we get a point of $S^2$.
 
These maps $S^2\to S^2$ are conformal, the coefficient
of dilatation of a transformation $g$ at a point $x$
is 
$$
k(g,x)=
(\alpha+x\gamma)^{-2}.
$$    
Clearly,  the coefficient satisfies the chain rule:
$$
k(g_1 g_2,x)=k(g_1,x)\, k(g_2,xg_1).
$$

{\sc Remark.}
For each $\lambda \in\C$ we  define a representation
of $\SO_0(1,3)$ in the space $C^\infty(S^2)$ by the formula
$$
T_\lambda\begin{pmatrix}
\alpha&\beta\\\gamma&\delta
\end{pmatrix} f(x)=
f\bigl((\alpha+x\gamma)^{-1}(\beta+x\delta)\bigr)\cdot (\alpha+x\gamma)^{-2\lambda}.
$$
By the chain rule,
$$
T_\lambda(g_1)\,T_\lambda(g_2)=T_\lambda(g_1g_2).
$$
The representations $T_\lambda$ are called {\it representations
of spherical principal series}.
If $\Re\lambda=1$, then the representations $T_\lambda$
are unitary in the Hilbert space $L^2(S^2)$.
\hfill $\boxtimes$

\sm

{\bf\punct The conformal group of the Riemann sphere and its Lie algebra.}
We consider the group $\SL(2,\C)$ consisting of complex 
matrices $\begin{pmatrix}a&b\\c&d \end{pmatrix}$
with determinant 1. 
By $\SU(2)$ we denote its subgroup consisting of unitary matrices,
elements of $\SU(2)$ have the form
$$
g=\begin{pmatrix}
a&b\\-\ov b&\ov a
\end{pmatrix}, \qquad\text{where $\det g=|a|^2+|b|^2=1$.}
$$

The group $\SL(2,\C)$ acts on the Riemann sphere 
$$\ov\C:=\C\cup \infty$$
 by conformal transformations
$$
g:\, z\mapsto z^{[g]}:=\frac{b+dz}{a+cz}.
$$
The transformation corresponding to the  matrix 
$\begin{pmatrix}-1&0\\0&-1\end{pmatrix}$
is trivial, so actually we have an action 
of the quotient group $\SL(2,\C)/\Z_2$.

We have
\begin{equation}
\wh d z^{[g]}=\frac{\wh d z}{(a+cz)^2\, \ov{(a+zc)}^2}.
\label{eq:transformation-dz}
\end{equation}

For $h\in \SU(2)$ 
we have
\begin{align}
1+\ov {z^{[h]}} u^{[h]}=\frac{1+\ov z u}{(\ov a- b \ov z)(a-\ov b u)},
\label{eq:1+zu}
\\
z^{[h]}-u^{[h]}=\frac{z-u}{(a-\ov b z)(a-\ov b u)}
\label{eq:z-u}
\end{align}

In particular, this implies that for  $h\in \SU(2)$
we have
\begin{equation}
\frac{\wh d z^{[h]}}{(1+ z^{[h]}\ov{z^{[h]}})^2}=
 \frac{\wh d z}{(1+z\ov z )^2},
\label{eq:whdzh}
\end{equation}
i.e., the transformations $z\mapsto z^{[h]}$ 
with $h\in \SU(2)$ preserve the measure \eqref{eq:measure}.

The stereographic projection $S^2\to\C$ identifies
conformal groups of $\ov\C$ and $S^2$, we have the surjective
homomorphism
$$
\SL(2,\C)\to \SO_0(1,3)
$$
sending
$\SU(2)\to \SO(3)$. The kernel in both cases consists of matrices 
$\begin{pmatrix}\pm 1&0\\0&\pm 1 \end{pmatrix}$.
The group $\SU(2)$ is a two-sheeted covering 
of  $\SO(3)$ and $\SL(2,\C)$ is a two-sheeted covering of  $\SO_0(1,3)$.

\sm 

For instance, we  have the following correspondences of 
one-parametric groups in $\SO(3)$ and in $\SU(2)$:
\begin{align*}
\begin{pmatrix}1&0&0\\
0&\cos 2\phi&\sin2\phi\\
0&-\sin 2\phi&\cos 2\phi
\end{pmatrix}&\longrightarrow
\begin{pmatrix}
\cos\phi& i \sin\phi
\\
i \sin\phi&\cos\phi
\end{pmatrix};
\\
\begin{pmatrix}
\cos2\phi&0&\sin 2\phi\\
0&1&0\\
-\sin 2\phi&0&\cos 2\phi
\end{pmatrix}&
\longrightarrow
\begin{pmatrix}
\cos \phi&\sin\phi\\
-\sin\phi&\cos\phi
\end{pmatrix};
\\
\begin{pmatrix}
\cos 2\phi&\sin 2\phi&0\\
-\sin 2\phi&\cos 2\phi&0\\
0&0&1
\end{pmatrix}
&
\longrightarrow
\begin{pmatrix}
e^{i\phi}&0\\
0&e^{-i\phi}
\end{pmatrix}.
\end{align*}
The corresponding vector fields on $\C$ are
respectively
\begin{gather*}
Q_1=i\Bigl((1-z^2)\frac{\partial}{\partial z}-
(1-\ov z^2)\frac{\partial}{\partial \ov z}\Bigr)
,\qquad
Q_2=(1+z^2)\frac{\partial}{\partial z}+
 (1+\ov z^2)\frac{\partial}{\partial \ov z},\\
Q_3= 2i\Bigl(-z \frac{\partial}{\partial z}+
 \ov z \frac{\partial}{\partial\ov z}\Bigr).
\end{gather*}
These vector fields form a basis in the Lie algebra $\mathfrak{su}(2)$
of the group $\SU(2)$. We prefer to work with the complexification
$\mathfrak{su}(2)_\C\simeq \mathfrak{sl}(2,\C)$ and choose the following basis
in $\mathfrak{su}(2)_\C$:
$$
E_-:=\frac12(Q_2-i Q_1),\qquad E_0:=\frac i2 Q_3,\qquad 
E_+=\frac12(Q_2+i Q_1),
$$
i.e.,
\begin{equation}
E_-:=\frac{\partial}{\partial z}+\ov z^2 \frac{\partial}{\partial \ov z},
\qquad 
E_0:=z\frac{\partial}{\partial z}-\ov z\frac{\partial}{\partial \ov z},\qquad
E_+:=z^2\frac{\partial}{\partial z}+\frac{\partial}{\partial \ov z}.
\label{eq:basis-E}
\end{equation}
They satisfy commutation relations
\begin{equation}
[E_0,E_-]=-E_-,\qquad [E_0,E_+]=E_+,\qquad [E_-,E_+]=2E_0.
\end{equation}

The $\SU(2)$-{\it invariant Laplacian} on $\C$ 
is
\begin{equation}
\DeltA=
\frac12 (E_+E_-+E_-E_+)- E_0^2=
\frac14(Q_1^2+Q_2^2+Q_3^2)=
(1+z\ov z)^2 \,\frac{\partial^2}{\partial z\, \partial \ov z},
\end{equation}
the $\DeltA$ is the usual Laplacian on the sphere $S^2$ written
in the stereographic coordinates.

The complexification of the {\it real} Lie algebra $\mathfrak{sl}(2,\C)$
is
$$
\mathfrak{sl}(2,\C)_\C\simeq \mathfrak{sl}(2,\C)\oplus \mathfrak{sl}(2,\C),
$$
The $\mathfrak{su}(2)_\C\simeq \mathfrak{sl}(2,\C)$ is
the diagonal subalgebra in the direct sum.
We define the following vector fields $\in \mathfrak{sl}(2,\C)_\C$
completing the collection \eqref{eq:basis-E} to the basis
in $\mathfrak{sl}(2,\C)_\C$:
\begin{equation}
F_-:=\frac{\partial}{\partial z}-\ov z^2 \frac{\partial}{\partial \ov z},
\qquad 
F_0:=z\frac{\partial}{\partial z}+\ov z\frac{\partial}{\partial \ov z},\qquad
F_+:=z^2\frac{\partial}{\partial z}-\frac{\partial}{\partial \ov z}.
\label{eq:basis-F}
\end{equation}
They satisfy relations
\begin{gather}
[E_-,F_-]= [E_0,F_0]=[E_+,F_+]=0;
\\
[E_0, F_-]=[E_-,F_0]=-F_-,\qquad [E_0, F_+]=[E_+,F_0]=F_+;
\\
 [E_-,F_+]=[E_+,F_-]=2 F_0
 \\
 [F_0,F_-]=-E_-,\qquad [F_0,F_+]=E_+,\qquad [F_-,F_+]=2E_0
\end{gather}
(so, we have a $\Z_2$-grading).

\sm

{\bf \punct Realizations of irreducible representations of $\SU(2)$.}
Recall a model for irreducible representations of $\SU(2)$
proposed in Berezin \cite{Ber}, \S 5 (see more details in this paper).
Fix $n=0$, 1, 2, \dots. We consider the space $\cB_m$ of
holomorphic functions $p(u)$ on $\C$
satisfying the condition
$$
\int_\C p(u)\,\ov{q(u)}\,(1+u\ov u)^{-m-2}\,\wh d u<\infty
.$$
This  space consists of
  polynomials of degree $\le m$.
We equip $\cB_m$  with the inner product
$$
\bigl\la p,q \bigr\ra_m:=
\frac {m+1}\pi\int_\C p(u)\,\ov{q(u)}\,(1+u\ov u)^{-m-2}\,\wh d u.
$$ 
The group $\SU(2)$ acts in the $(m+1)$-dimensional
 space $\cB_m$ by the formula
$$
\tau_m\begin{pmatrix}
a&b\\-\ov b&\ov a
\end{pmatrix} p(u)=
p\Bigl(\frac{b+\ov a z}{a-\ov b u}\Bigr)\,(a-\ov b u)^m,
$$
these operators are unitary. The Lie algebra
$\mathfrak{su}(2)_\C$
acts by operators
\begin{equation}
E^{(u)}_-:=\frac{\partial}{\partial u},\qquad
E^{(u)}_0:=u \frac{\partial}{\partial u}+\frac{m}{2},
\qquad E^{(u)}_+:=u^2\frac{\partial}{\partial u}+m u.
\end{equation}

Notice that vectors $1$, $u$, \dots, $u^m\in\cB_m$
are pairwise orthogonal and
\begin{equation}
\|u^k\|_m^2=\frac{k!\, (m-k)!}{m!}.
\label{eq:norms-monomials}
\end{equation}
Indeed,
\begin{multline*}
\|u^k\|^2_m
=\frac{m+1}\pi\int_\C |u|^{2k}(1+|u|^2)^{-m-2}\wh d u
=\frac{2\pi(m+1)}{\pi}\int_0^\infty \frac{r^{2k}\,r\,dr}{(1+r^2)^{m+2}}=
\\=
(m+1)
\int_0^\infty \frac{x^k\,dx}{(1+x)^{m+2}}=(m+1)\,B(k+1,m-k+1)=
\frac{k!\, (m-k)!}{m!}.
\end{multline*}
Here and below we use the following standard beta-integral (see for example \cite{AAR},
formula (1.1.20)):
$$
\int_0^\infty \frac{x^{a-1}dx}{(1+x)^{a+b}}=B(a,b).
$$

{\sc Remark.} Formula \eqref{eq:norms-monomials}
implies that the Hilbert space $\cB_m$ is determined by
the reproducing kernel 
$$L(u,w)=(1+u\ov w)^m$$
on $\C$.
In other words, for any $a\in\C$ consider the function
$\Phi_a(u):=L(u,a)$. Then 
$$\qquad\qquad\qquad\qquad\qquad\qquad\quad
p(a)=\bigl\la  p, \Phi_a \bigr\ra_{\cB_m}
\qquad\qquad\qquad\qquad\qquad\qquad\quad
$$
for any $p\in \cB_m$.
\hfill $\boxtimes$

\sm

{\bf \punct Intertwining operators.}
It is well-known that the representation of 
$\SO(3)$ in $L^2(S^2)$ is a direct sum
$$
L^2(S^1)=\bigoplus_{n=0}^\infty \cH_n
$$
of $(2n+1)$-dimensional irreducible representations 
of $\SO(3)$. Subspaces $\cH_n$ are spaces 
of eigenfunctions $\psi$ of the spherical Laplacian,
$$
\DeltA\psi=-n(n+1)\psi.
$$
The functions
$$
\psi_j^{(n)}:=
E_+^j \frac{\ov z^n}{(1+z \ov z)^n}
=\Bigl(z^2\frac{\partial}{\partial z}+\frac{\partial}{\partial \ov z}\Bigr)^j 
\frac{\ov z^n}{(1+z \ov z)^n},\quad \text{where $j=0$, \dots, $2n$,}
$$
form an orthogonal basis in $\cH_n$.
These functions also are eigenfunctions of $E_0$,
$$
E_0\, \psi_j^{(n)}=(-n+j) \,\psi_j^{(n)}.
$$

 The representations of $\SO(3)$ in $\cH_n$ are
equivalent to
the representations $\tau_{2n}$ of $\SU(2)/\Z_2$ defined 
in the previous subsection.

\sm

Denote by $K(z, \ov z; u,n)$
the kernel
\begin{equation}
K(z, \ov z; u,n):=
\Bigl(\frac{(1+\ov z u)(z-u}{1+z\ov z}\Bigr)^n.
\end{equation}

\begin{proposition}
\label{pr:intertwiner}
For a fixed $n$ the operator
\begin{equation}
A_n f(u):= \frac1\pi\int_\C K(z, \ov z; u,n)\, 
f(z,\ov z)\frac{\wh d z}{(1+z\ov z)^2}
\label{eq:A-n}
\end{equation}
is an $\SU(2)$-intertwining operator
$$
L^2\bigl(\C,(1+z\ov z)^{-2} \wh d z\bigr)\to \cB_{2n}
.$$
\end{proposition}

{\sc Proof.} Fix $n$.
First, let us check that the integral is well defined. 
 The expression 
$$
K(\cdot)=(\dots)^n=\sum_{j=0}^{2n}
c_j(z,\ov z)  u^j
$$
is a polynomial in $u$ of degree $2n$
with coefficients of the form
$$
c_j(z,\ov z)= \frac{\kappa_j(z,\ov z)}
{(1+z\ov z)^n}
,$$
where $\kappa_j(z,\ov z)$ are polynomials
of degree $2n$. So, the coefficients $c_j(z,\ov z)$
are bounded functions, they are contained in our $L^2$.
Hence, for $f\in L^2$,
$$
A_n f(u)=\frac1\pi\sum_{j=0}^{2n} u^j\int_\C c_j(z,\ov z)\, f(z,\ov z)
\frac{\wh d z}{(1+z\ov z)^2}=\frac1\pi\sum_{j=0}^{2n} u^j\cdot
\bigl\la\, c_j,\ov {\phantom{.}f \phantom{.}}\, \bigr\ra_{L^2(\dots)}.
$$
The inner products are well-defined,  therefore we get a well-defined operator $L^2\to \cB_{2n}$.

\sm

Keeping in  mind \eqref{eq:1+zu}--\eqref{eq:z-u}, we get
that for $h\in\SU(2)$
\begin{equation}
K\bigl(z^{[h]}, \ov {z^{[h]}};u^{[h]},n\bigr)
=\frac{ K(z,\ov z; u,n)}{(a-\ov b u)^{2n}} 
.
\label{eq:KK}
\end{equation}
Consider the composition of the transformation
$f(z,\ov z)\mapsto f\bigl(z^{[h]}, \ov{z^{[h]}}\bigr)$
and the operator $A_n$.
Keeping in the mind   \eqref{eq:whdzh} and \eqref{eq:KK},
we get
\begin{multline*}
\frac1\pi\int_\C K(z,\ov z;u,n)\, f\bigl(z^{[h]}, \ov{z^{[h]}}\bigr)\,
\frac{\wh d z}{(1+z \ov z)^2}
=\\=
\frac1\pi\int_\C K(z,\ov z;u,n)\,f\bigl(z^{[h]}, \ov{z^{[h]}}\bigr)
\,
\frac{\wh d z^{[h]}}{(1+z^{[h]} \ov {z^{[h]}})^2}
=\\=
\frac1\pi\int_\C 
K\bigl(z^{[h]} ,\ov {z^{[h]}} ;u^{[h]} ,n\bigr)\,
(a-\ov b u)^{2n}\,
f\bigl(z^{[h]}, \ov{z^{[h]}}\bigr)\,
\frac{\wh d z^{[h]}}{(1+z^{[h]} \ov {z^{[h]}})^2}.
\end{multline*}
We substitute $\xi=z^{[h]}$ and come to
$$
\qquad
\biggl(\frac1\pi\int_\C 
K\bigl(\xi ,\ov \xi ;u^{[h]} ,n\bigr)\,
f(\xi, \ov \xi)
\frac{\wh d \xi}{(1+\xi \ov \xi)^2}\biggr)
\cdot (a-\ov b u)^{2n}=
\tau_{2n} \, A_n f(u).
\qquad
\square
$$

\sm 

{\bf \punct The transformation $\boldsymbol{\{A_n\}}$
in the space  $\boldsymbol{L^2(S^2)}$.}

\begin{lemma}
\label{l:norms}
Let $\psi\in \cH_n$.
Then 
$$
\|A_n\psi\|^2_{\cB_{2n}}=\frac{n!\,n!}{(2n+1)!}\, \|\psi\|^2_{\cH_n}.
$$
\end{lemma}

{\sc Proof.}
The operator $A_n:\cH_n \to\cB_{2n}$   intertwines
two irreducible unitary representations; by the Schur lemma,
the ratio $\|A_n\psi\|^2/\|\psi\|^2$ does not depend on $\psi$.
Let us evaluate it for  the function
$\psi_0^{(n)}(z,\ov z)=\ov z^n/(1+z \ov z)^n\in \cH_n$.

We have
\begin{multline*}
\|\psi_0^{(n)}\|^2_{L^2(\C,(1+z\ov z)^{-2}\wh d z)}=
\frac 1\pi \int_\C \frac{\ov z^n}{(1+z\ov z)^n}
\cdot \frac{z^n}{(1+z\ov z)^n}\cdot \frac{\wh d z}{(1+z\ov z)^2} 
=\\=
\frac{2\pi}{\pi}\int_0^{\infty}
\frac{r^{2n}\cdot rdr}{(1+r^2)^{2n+2}}=\int_0^\infty\frac{\rho^n d\rho}{(1+\rho)^{2n+2}}=B(n+1,n+1)=\frac{n!\,n!}{(2n+1)!}.
\end{multline*}
Next,
\begin{multline*}
A_n \psi_0^{(n)}(u)
=\frac 1\pi \int_\C \Bigl(\frac{(1+\ov z u)(z-u}{1+z\ov z}\Bigr)^n
\cdot
  \frac{\ov z^n}{(1+z\ov z)^n}\cdot \frac{\wh d z}{(1+z\ov z)^2} 
  =\\=
 \frac 1\pi\int_0^\infty \int_0^{2\pi}
 \frac{\bigl[1+u r e^{-i\phi} \bigr]^n \, \bigl[r e^{i\phi}-u\bigr]^n\, r^n e^{-i\phi n}\,r\,d\phi\,dr}{(1+r^2)^{2n+2}} 
 =\\=
  \frac 1\pi\int_0^\infty \int_0^{2\pi}
 \frac{\bigl[1+u r e^{-i\phi} \bigr]^n \, \bigl[r -u e^{-i\phi}\bigr]^n\, r^n \,r\,d\phi\,dr}{(1+r^2)^{2n+2}} .
\end{multline*}
We open brackets
$[\dots]^n [\dots]^n$
and get
$$
[\dots]^n [\dots]^n=r^n+\sum_{k=1}^{2n} p_k(u,r) e^{-ik\phi}.
$$
After integration $\int_0^{2\pi}$ all summands except the first one vanish
and we get
$$
2\int_0^\infty\frac{r^n\cdot r^n\,r\,dr}{(1+r^2)^{2n+2}}=
\frac{n!\,n!}{(2n+1)!}= \frac{n!\,n!}{(2n+1)!}\cdot u^0.
$$
By \eqref{eq:norms-monomials}, we have
$$
\|u^0\|^2_{\cB_{2n}}=1.
$$
Therefore,
$$
\|A_m \psi_0^{(n)}\|_{\cB_{2n}}^2=\Bigl(\frac{n!\,n!}{(2n+1)!}\Bigr)^2,
$$
and we come do the desired statement.
\hfill $\square$ 

\sm

Thus, for a function $f$ on the sphere $S^2$
we have a sequence
$$
(\frf_0,\frf_1,\frf_2,\dots)=
(A_0 f,A_1 f, A_2 f,\dots), 
$$
where $\frf_j=A_j f\in \cB_{2n}$.
We regard  such sequences as functions $\frF(u,n)$
depending on a nonnegative integer $n$ and a complex variable
$u$; for a given $k$ the degree 
of a polynomial $\frF(n,u)\Bigl|_{n=k}$ is $\le 2k$.

Clearly, the image of $L^2(S^2)$ 
consists of functions $\frF(u,n)$ satisfying
$$
\sum_{k\ge 0} \frac{(2k+1)!}{k!\,k!}\,\,
\Bigl\| \frF\bigl|_{n=k} \Bigr\|^2_{\cB_{2k}}<\infty.
$$
It
is a Hilbert with inner product defined by 
$$
\bigl\la \frF,\frF'\bigr\ra=\sum_{k=0}^\infty 
 \frac{(2k+1)!}{k!\,k!}\,\,
  \Bigl\la \frF\bigl|_{n=k},\frF'\bigl|_{n=k} \Bigr\ra_{\cB_{2k}}
  .
$$
By the Stirling formula, the asymptotics of coefficients
in this formula is
\begin{equation}
\frac{(2k+1)!}{k!\,k!}
=
(2k+1)\cdot \frac{(2k)!}{k!\,k!}
 \sim 2k\cdot \frac{2^{2k}}{\sqrt {\pi k}}
\qquad
\text {as $k\to+\infty$.}
\label{eq:stirling}
\end{equation}

{\bf \punct The transformation $\boldsymbol{\{A_n\}}$
in the space  $\boldsymbol{C^\infty(S^2)}$.%
\label{ss:image-C-infty}}
A function
$$
F=\sum_{n=0}^\infty \psi_n, \qquad \text{where $\psi_n\in \cH_n$}
$$
is $C^\infty$-smooth if and if for any%
\footnote{The operator $(1-\DeltA)$ acts in
each $\cH_n$ as a multiplication by $1+n(n+1)$.
It is an elliptic positive definite  operator of order 2, therefore its powers
$(1-\DeltA)^{t/2}$, where $t\in \R$, are continuous invertible
pseudo-differential operators from $L^2(S^2)$ to the Sobolev space $W_2^{-t}(S^2)$,
see for example \cite{Tay}, Section II.6.
The space $C^\infty(S^2)$ is $\cap_{t>0} W_2^{t}(S^2)$.} $L>0$
the following estimate holds
$$
\|\psi_n\|_{L^2(S^2)}=o(n^{-L}) \quad \text {as $n\to+\infty$.}
$$
By \eqref{eq:stirling}, the image $\cW^\infty$ of $C^\infty(S^2)$
consits of all functions $\frF(u,n)$ such that for all $L>0$ 
$$
\Bigl\|\frF\bigr|_{n=k} \Bigr\|_{\cB_{2k}}=o(2^{-k}k^{-L})\qquad \text {as $k\to+\infty$.}
$$



\section{Calculations}

\COUNTERS

{\bf \punct Smoothness of the kernel.}
In the proof of Proposition \ref{pr:intertwiner}
we show that the operators $A_n$ are well-defined. This argument
is not sufficient for differentiation of the integral 
\eqref{eq:A-n} with respect to conformal vector fields.
However, the singularity at $\infty$ is absent, actually
our expression is smooth on the sphere  $S^2$.

To observe the smoothness, we represent the expression in the big bracket
in \eqref{eq:K-x} near the point $x=(0,0,1)$
as
\begin{equation*}
\frac12\Bigl[\frac{u(x_1^2+x_2^2)}{1-x_3}+ \Bigl(-u^2(x_1-i x_2)+(x_1+i x_2)\Bigr)
-(1-x_3) u\Bigr].
\end{equation*}
The second and the third summands in the brackets $[\dots]$
 are smooth on $S^2$. 
In the first summand, we
eliminate the irrationality in the denominator:
$$
\frac{u(x_1^2+x_2^2)(1+x_3)}{(1-x_3)(1+x_3)}=
\frac{u(x_1^2+x_2^2)(1+x_3)}{x_1^2+x_2^2}= u(1+x_3)
,
$$
and we again get a smooth expression.

\sm

{\bf \punct The verification of Theorem \ref{th:1}.}
First, we verify \eqref{eq:drob1}. Denote by 
$$
\cL:=-\frac1{2(n+1)(2n+1)}\, T_+\, \frac{\partial^2}{\partial u^2}
+\frac{n}{2(2n+1)}\,T_-
$$
 the operator
in the right side of \eqref{eq:drob1}. Our statement
is equivalent to the identity
$$
\cL \,K(z,\ov z; u, n)=\frac{\ov z}{1+z\ov z}\cdot K(z,\ov z; u, n).
$$
Now it can be verified by any system of computer algebra%
\footnote{The author used Wolfram Mathematica.}. We wish to explain how to find this identity and how to verify it by hands.

Let us transform to a convenient for us form the expressions
$$
T_- K,\qquad \frac{\partial}{\partial u} T_+,
\qquad \frac{\partial^2}{\partial u^2} T_+.
$$

Denote
$$
\kappa:=\frac{(1+\ov z u)(z-u)}{1+z\ov z},
$$
so $K=\kappa^n$.
Notice that 
$$
\kappa^{-1}=
\frac{1+z\ov z}{(1+\ov z u)(z-u)}=\frac{\ov z}{1+ \ov z u}+
\frac1{z-u}.
$$
Therefore,
\begin{equation}
T_- K=\Bigl(\frac{\ov z}{1+ \ov z u}+
\frac1{z-u}\Bigr) \cdot K.
\label{eq:T-K}
\end{equation}

Applying the formula
\begin{equation}
\frac{\partial}{\partial x}\prod_j f_j(x)^{m_j} 
=
\Bigl(\sum_j \frac{m\,  f'_j(x) }{f_j(x)} \Bigr)
\cdot
\prod_j f_j(x)^{m_j}, 
\label{eq:logarithmic}
\end{equation}
we get
\begin{multline}
\frac{\partial^2}{\partial u^2}\, T_+ K
=\frac{\partial^2}{\partial u^2} \kappa^{n+1}=
\frac{\partial}{\partial u}\Bigl[(n+1)\Bigl(\frac{\ov z}{1+\ov z u}-\frac{1}{z-u} \Bigr) \kappa^{n+1}\Bigr]
=\\=
\Bigl\{(n+1)^2\Bigl(\frac{\ov z}{1+\ov z u}-\frac{1}{z-u} \Bigr)^2
-(n+1)\Bigl(\frac{\ov z^2}{(1+\ov z u)^2}+\frac{1}{(z-u)^2} \Bigr)\Bigl\}
\,\kappa^{n+1}
=\\=
\Bigl\{n(n+1)\Bigl(\frac{\ov z^2}{(1+\ov z u)^2}+\frac{1}{(z-u)^2} \Bigr)-2(n+1)^2 \frac{\ov z}{(1+\ov z u)(z-u)}
\Bigr\}\kappa\cdot \kappa^n.
\end{multline}
We transform summands in the right hand side,
$$
\frac{\ov z}{(1+\ov z u)(z-u)}\cdot \kappa=\frac{\ov z}{1+z \ov z};
$$
$$
\frac{\ov z^2}{(1+\ov z u)^2}\cdot \kappa=
\frac{\ov z^2(z-u)}{(1+\ov z u)(1+\ov z z)}
=\frac{\ov z}{1+\ov z u} - \frac{\ov z}{1+\ov z z};
$$
$$
\frac{1}{(z-u)^2}\cdot \kappa =\frac{1+\ov z u}{(z-u)(1+\ov z z)}
=\frac{1}{z-u}-\frac{\ov z}{1+\ov z z}.
$$
Thus, we have
\begin{equation}
\frac{\partial^2}{\partial u^2}\, T_+ K
=\Bigl\{
n(n+1)\Bigl(\frac{\ov z}{1+\ov z u}+ \frac{1}{z-u}\Bigr)
+ \frac{C\ov z}{1+z\ov z}
\Bigr\}\cdot K,
\label{eq:partial2K}
\end{equation}
where $C=-2n(n+1)-2(n+1)^2$. Comparing the last expression
with \eqref{eq:T-K} we observe that
$$
u^2\frac{\partial^2}{\partial u^2}\,T_+\, K-2n(n+1) \,T_- K=
-2(n+1)(2n+1)\cdot \frac{\ov z}{1+z \ov z}\cdot K.
$$
This establishes \eqref{eq:drob1}.

\sm

Next,
$$
\Bigl[z^2\frac{\partial}{\partial z}+\frac{\partial}{\partial \ov z},\,\,
\frac{\ov z}{1+z \ov z}\Bigr]=\frac{1 - z\ov z}{1+z \ov z}.
$$
We know the operator corresponding $z^2\frac{\partial}{\partial z}+\frac{\partial}{\partial \ov z}$  (see \eqref{eq:triv3}),
and we obtained the operator $\cL$ corresponding to $\frac{\ov z}{1+z \ov z}$
(see \eqref{eq:triv3}). Evaluating their commutator
$$
\Bigl[u^2 \frac{\partial}{\partial u}+2nu,\,\,
-\frac1{2(n+1)(2n+1)}\, T_+\, \frac{\partial^2}{\partial u^2}
+\frac{n}{2(2n+1)}\,T_-\Bigr],
$$
we come to  formula \eqref{eq:drob2}.

\sm

Repeating the same reasoning with the commutator
$$
\Bigl[z^2\frac{\partial}{\partial z}+\frac{\partial}{\partial \ov z},
\,\,
\frac{1-z\ov z}{1+z \ov z}\Bigr]=-\frac{2 z}{1+z \ov z},
$$
after straightforward calculations
come to \eqref{eq:drob3}.

\sm

{\bf \punct Integration by parts.}
We say that an operator $Q$ in the space of functions
on $\C$ is {\it formally transposed} to an operator $R$
if
$$
\int_\C  \bigl(Qf(z,\ov z)\bigr)\cdot g(z,\ov z)
\,(1+z \ov z)^{-2}  \wh d z=
\int_\C  f(z,\ov z)\cdot \bigl(Rg(z,\ov z)\bigr) \,
(1+z \ov z)^{-2}  \wh d z
$$
for all $f$, $g$, which are actually are smooth on the sphere $S^2$.
Integrating by parts, 
\begin{multline}
\int_\C  \Bigl(\frac{\partial}{\partial z} f(z)\Bigr) g(z,\ov z)
\frac{\wh d z}{(1+z \ov z)^2}
=-\int_\C f(z,\ov z) \cdot \frac{\partial}{\partial z}\bigl( g(z,\ov z)
(1+z \ov z)^{-2} \bigr) \wh d z
=\\=
-\int_\C f(z,\ov z) \cdot \Bigl(\frac{\partial}{\partial z}-
\frac{2\ov z}{1+z\ov z}\Bigr)\, g(z,\ov z)\cdot
\frac{\wh d z}{(1+z \ov z)^{2}}
,
\end{multline}
we get
$$
\Bigl(\frac{\partial}{\partial z}\Bigr)^t
=-\frac{\partial}{\partial z}+ \frac{2\ov z}{1+z \ov z},
\qquad \Bigl(\frac{\partial}{\partial \ov z}\Bigr)^t
=-\frac{\partial}{\partial \ov z}+ \frac{2 z}{1+z \ov z}.
$$
For operators \eqref{eq:basis-E},
we have 
$$
E_\pm^t=-E_\pm^t, \qquad E_0^t=-E_0.
$$
For operators
\eqref{eq:basis-F}, straightforward calculations show
\begin{equation}
F_-^t=\Bigl(\frac{\partial}{\partial z}-
\ov z^2 \frac{\partial}{\partial \ov z} \Bigr)^t
=-F_-+\frac{4 \ov z}{1+z\ov z};
\label{eq:FF}
\end{equation}
$$
F_0^t=\Bigl(z\frac{\partial}{\partial z}+
\ov z \frac{\partial}{\partial \ov z} \Bigr)^t=
-F_0-2\frac{1-z \ov z}{1+z\ov z};
$$
$$
F_+^t=\Bigl(z^2\frac{\partial}{\partial z}-
 \frac{\partial}{\partial \ov z} \Bigr)^t=
-F_+-\frac{4z}{1+z\ov z}.
$$

{\bf \punct The proof of Theorem \ref{th:2}.}
Let us verify the statement \eqref{eq:conf1}.
By \eqref{eq:FF}, we get
\begin{multline*}
\int_\C K(z,\ov z;u,n)\, F_- g(z,\ov z)\, \frac{\wh d z}{(1+z\ov z)^2}
=\\=
\int_\C\Bigl[\Bigl(-F_-+\frac{4\ov z}{1+z\ov z}\Bigr)
K(z,\ov z;u,n)\Bigr]\,  g(z,\ov z)\, \frac{\wh d z}{(1+z\ov z)^2}.
\end{multline*}
So, we need a differential-difference  operator $\cR$
such that
$$
\Bigl(-F_-+\frac{4\ov z}{1+z\ov z}\Bigr)
K(z,\ov z;u,n) =\cR\, K(z,\ov z;u,n).
$$
By \eqref{eq:drob1}, it is sufficient to find $\cR'$
such that
$$-F_-\,K(z,\ov z;u,n) =\cR'\, K(z,\ov z;u,n).$$
It  to verify that
\begin{equation}
F_-\, K=n\Bigl(\frac{\ov z}{1+\ov z u}+\frac1{z-u}-
\frac{\ov z}{1+z\ov z}\Bigr)\, K.
\label{eq:FK}
\end{equation}
This identity can be found by manipulations with the formula 
\eqref{eq:logarithmic}.

We compare expressions
for
\begin{equation}
K^{-1}\cdot T_- K,\qquad K^{-1}\cdot \frac{\partial^2}{\partial u^2}T_+K,
\qquad K^{-1}\cdot F_-\, K
,
\label{eq:expressions}
\end{equation}
see \eqref{eq:T-K}, \eqref{eq:partial2K}, \eqref{eq:FK}.
Each expression has form
$$
A_j \Bigl(\frac{\ov z}{1+\ov z u}+\frac1{z-u}\Bigr)+ B_j\cdot \frac{\ov z}{1+z\ov z},
$$
where $A_j$, $B_j$ are polynomials in $n$. Therefore 
\eqref{eq:expressions} are linearly dependent with coefficients depending on $n$.
This allows to express $F_-\, K$ as a linear combination
of $T_- K$ and $\frac{\partial^2}{\partial u^2}T_+K$.


Next, $[E_+,F_-]=2 F_0$. Since we know corresponding
$E_+$ and $F_-$, after straightforward calculations we  evaluate the operator \eqref{eq:conf2} corresponding to $F_0$.

Finally, $[E_+,F_0]=F_+$, and this allow to evaluate \eqref{eq:conf3}.
 
\sm

{\bf \punct Proof of Corollary \ref{cor:}.%
\label{ss:proof-corollary}}
We begin with three simple observations.

\begin{observation}
\label{obs:1}
The space  $\C[\cS]$
is closed with respect to the product.
\end{observation}

Clearly, $\Theta^+_{m,k}\, \Theta^+_{n,l}=\Theta^+_{m+n,k+l}$.
Let us examine products
$\Theta^+_{m,k}\cdot \Theta^-_{n,l}$.
 To be definite, assume
that $k\ge l$. Then
\begin{multline*}
\Theta^+_{k,m}\cdot \Theta^-_{l,n}=
\frac{(z\ov z)^l z^{k-l}}{(1+z \ov z)^{m+n}}=
\frac{\bigl[(1+z\ov z)-1\bigr]^l z^{k-l}}{(1+z \ov z)^{m+n}}
=\\=
 \frac{\sum_{j=0}^l C_l^j(-1)^{j}(1+z\ov z)^j\cdot z^{k-l}}
{(1+z \ov z)^{m+n}}=
\sum_{j=0}^l C_l^j(-1)^{j} \Theta^+_{m+n-j,k-l}.
\end{multline*}

\begin{observation}
\label{obs:2}
 The space $\C[\cS]$ is invariant with respect to  the operators
$\frac{\partial}{\partial z}$, $z\frac{\partial}{\partial z}$, 
$z^2\frac{\partial}{\partial z}$, $\frac{\partial}{\partial \ov z}$, $\ov z\frac{\partial}{\partial \ov z}$, 
$\ov z^2\frac{\partial}{\partial \ov z}$.
\end{observation}

\sm

This follows from the following trivial calculations 
(but their result is not a priory obvious). Let $0\le m \le n$. We apply our operators
to $\Theta^+_{n,n-m}$:
$$
z^2 \frac{\partial}{\partial z} \frac{z^{n-m}} {(1+z \ov z)^n}
= \frac{n z^{n-m+1}}{(1+z\ov z)^{n+1}}-
\frac{m z^{n-m+1}}{(1+z\ov z)^{n}}\,\,
=\,\, n\Theta^+_{n+1,n-m+1}-m\Theta^+_{n,n-m+1}.
$$
If $m=0$, we can not write $\Theta^+_{n,n-m+1}$,
but in this case the term with $\Theta^+_{n,n-m+1}$
is absent.
 Next,
$$
z \frac{\partial}{\partial z} \frac{z^{n-m}} {(1+z \ov z)^n}
\!=\! \frac{n z^{n-m}}{(1+z\ov z)^{n+1}}-
\frac{m z^{n-m}}{(1+z\ov z)^{n}},
\,\, 
 \frac{\partial}{\partial z} \frac{z^{n-m}} {(1+z \ov z)^n}
\!=\! \frac{n z^{n-m-1}}{(1+z\ov z)^{n+1}}-
\frac{m z^{n-m-1}}{(1+z\ov z)^{n}}.
$$
The right hand side of the second equality does not have the desired form if $m=n$. But in that case $\frac{\partial}{\partial z} (1+z \ov z)^{-n}=-n \ov z (1+z \ov z)^{-n-1}$. Further, 
$$
\frac{\partial}{\partial \ov z}\,
\frac{z^{n-m}}
{(1+z \ov z)^{n}}=\frac{n\,z^{n-m+1}}
{(1+z \ov z)^{n+1}},\quad 
\ov z\frac{\partial}{\partial \ov z}\,
\frac{z^{n-m}}
{(1+z \ov z)^{n}}=\frac{n\,z^{n-m}}{(1+z \ov z)^{n+1}}-\frac{n\,z^{n-m}}{(1+z \ov z)^{n}};
$$
$$
\ov z^2\frac{\partial}{\partial \ov z}\,
\frac{z^{n-m}}
{(1+z \ov z)^{n}}=-\frac{n z^{n-m-1}}{(1+z \ov z)^{n+1}}+
\frac{2n z^{n-m-1}}{(1+z \ov z)^{n}}-
\frac{n z^{n-m-1}}{(1+z \ov z)^{n-1}}.
$$
The last transformation is not appropriate for us
if $m=n$. In this case,
$$
\ov z^2\frac{\partial}{\partial \ov z}\,
\frac{1}
{(1+z \ov z)^{n}}=-\frac{n z \ov z^2}{(1+z \ov z)^{n+1}}
=\frac{n \ov z}{(1+z \ov z)^{n+1}}- \frac{n \ov z}{(1+z \ov z)^{n}}.
$$ 

\begin{observation}
\label{obs:3}
 In the coordinates $x_1$, $x_2$, $x_3$ the space $\C[\cS]$ is the space of all polynomials. 
\end{observation}

By \eqref{eq:x(z)}, the coordinates $x_1$, $x_2$, $x_3$
as functions on $(z,\ov z)$ 
are contained in $\C[\cS]$. On the other hand,
\begin{multline}
\frac{z^k}{(1+z\ov z)^n}=\Bigl(\frac{z}{1+z\ov z}\Bigr)^k\cdot
 \frac1{2^{n-k}}\cdot
\Bigl(1+\frac{1-z\ov z}{1+z\ov z}\Bigr)^{n-k}=\\=
(x_1+i x_2)^k\cdot \frac1{2^{n-k}}\cdot (1-x_3)^{n-k},
\label{eq:first-line}
\end{multline}
and this implies our observation.

We also continue \eqref{eq:first-line} as
\begin{equation}
\frac{z^k}{(1+z\ov z)^n}= \frac1{2^{n-k}}
\Bigl(\frac{z}{1+z\ov z}\Bigr)^k\cdot\sum_{j=0}^{n-k} C_{n-k}^j \Bigl(\frac{1-z\ov z}{1+z\ov z}\Bigr)^j.
\label{eq:Theta-in-space}
\end{equation}

Now we are ready to verify Corollary \ref{cor:}.

\sm

{\sc Proof of statement a).} Products of operators of the form
\eqref{eq:diff-diff} have the same form \eqref{eq:diff-diff}. Therefore it is sufficient to verify the statement for
all factors in \eqref{eq:NEBASIS}. 

For $\Theta_{n,k}^\pm$ we
take its decomposition \eqref{eq:Theta-in-space}
and refer to \eqref{eq:drob1}--\eqref{eq:drob3}.

For $\frac{\partial}{\partial z}$ we write
$$
\frac{\partial}{\partial z}=\frac12\Bigl[
\Bigl(\frac{\partial}{\partial z}+ 
\ov z^2\frac{\partial}{\partial \ov z}\Bigr)
+
\Bigl(\frac{\partial}{\partial z}-
\ov z^2\frac{\partial}{\partial \ov z}\Bigr)
\Bigr],
$$
and apply \eqref{eq:conf1} and \eqref{eq:drob1}, etc.
\hfill $\square$

\sm

{\sc Proof of  the statement b).}
Recall that 
\begin{equation}
\Bigl[\frac{\partial}{\partial z},z\frac{\partial}{\partial z}  \Bigr]=\frac{\partial}{\partial z},
\quad
\Bigl[\frac{\partial}{\partial z},z^2\frac{\partial}{\partial z}  \Bigr]=2z\frac{\partial}{\partial z},\quad
\Bigl[z\frac{\partial}{\partial z},z^2\frac{\partial}{\partial z}  \Bigr]=z^2\frac{\partial}{\partial z}.
\label{eq:last}
\end{equation}
We have the same relations for differentiations in $\ov z$;
differentiations in $z$ and $\ov z$ commute.

 Consider a product of two
expressions \eqref{eq:NEBASIS}
and try to transform it to the same order of factor
using transpositions of the form
$$
AB=BA+[A,B].
$$  
Generally, after such transposition we get  additional summands,
but $[A,B]$ always has the form $\sum_j c_j \Theta^{\delta_j}_{n_j,k_j}$, 
$\frac{\partial}{\partial z}$, \dots,
 $\ov z^2\frac{\partial}{\partial \ov z}$
 (this follows from Observation \ref{obs:2} and \eqref{eq:last}).
 The number of factors in an additional summand 
 is smaller than in initial summand.  
 Keeping in mind Observation \ref{obs:1},we observe the process leads to a sum whose summands have the form
 \eqref{eq:NEBASIS}).
\hfill $\square$

\tt

University of Graz,
\\
\phantom{.}
\hfill Department of Mathematics and Scientific computing;

Higher School of Modern Mathematics MIPT,


Moscow State University, MechMath. Dept;

 University of Vienna, Faculty of Mathematics.
 
 \sm

e-mail:yurii.neretin(dog)univie.ac.at

URL: https://www.mat.univie.ac.at/$\sim$neretin/ 

\phantom{URL:} https://imsc.uni-graz.at/neretin/index.html

\end{document}